\documentclass[12pt]{article}
\let\today\relax
\makeatletter
\def\ps@pprintTitle{%
	\let\@oddhead\@empty
	\let\@evenhead\@empty
	\def\@oddfoot{}%
	\let\@evenfoot\@oddfoot}
\makeatother
\usepackage{graphicx}
\usepackage{amssymb}
\usepackage{showlabels}
\usepackage{epsf}
\usepackage{amsbsy,amsmath}
\usepackage{mathtools}
\usepackage{mathrsfs}
\usepackage{amsfonts}
\usepackage{amssymb}
\usepackage{enumitem}
\usepackage{eucal}
\usepackage{cases}
\usepackage{graphics,mathrsfs}
\usepackage{amsthm}
\usepackage{secdot}
\usepackage{esint}
\usepackage{hyperref}
\usepackage{varwidth}
\usepackage{tasks}
\usepackage{geometry}
\newtheorem{theorem}{Theorem}[section]
\newtheorem{lemma}[theorem]{Lemma}

\theoremstyle{definition}
\newtheorem{definition}[theorem]{Definition}

\theoremstyle{remark}
\newtheorem{remark}[theorem]{Remark}
\numberwithin{equation}{section}

\numberwithin{equation}{section}
\usepackage{xcolor}

\begin{document}
% \begin{frontmatter}
	\title{\sc Elliptic problem driven by different \\ types of nonlinearities}
%\tnotetext[label1]{}
	\author{\sc Debajyoti Choudhuri$^a$ and Du\v{s}an D.  Repov\v{s}$^{b}$\footnote{Corresponding
		author: dusan.repovs@guest.arnes.si}\\
	\small{$^a$Department of Mathematics, National Institute of Technology Rourkela,}\\ 
	\small{Rourkela,  769008, Odisha, India.}\\	 
	\small{$^b$Faculty of Education and Faculty of Mathematics and Physics,}\\
	\small{University of Ljubljana, Ljubljana, 1000, Slovenia.}\\
	\small{\it Email: dc.iit12@gmail.com \& dusan.repovs@guest.arnes.si}}
\date{\today}
\maketitle
\begin{abstract}
\noindent In this paper we establish the existence and multiplicity of nontrivial solutions to the following problem 
\begin{align*}
\begin{split}
(-\Delta)^{\frac{1}{2}}u+u+(\ln|\cdot|*|u|^2)&=f(u)+\mu|u|^{-\gamma-1}u,~\text{in}~\mathbb{R},
\end{split}
\end{align*}
where $\mu>0$, $(*)$ is the {\it convolution} operation between two functions, $0<\gamma<1$, $f$ is a function with a certain type of growth. We prove the existence of a nontrivial solution at a certain mountain pass level and another ground state solution when the nonlinearity $f$ is of exponential critical growth. 
	\begin{flushleft}
		{\it Keywords}:~  Fractional Laplacian, ground state solution, singularity.\\
		{\it Math. Subj. Classif. (2010)}:~35R11, 35J75, 35J60, 46E35.
	\end{flushleft}
\end{abstract}
%\end{frontmatter}
%\linenumbers
\section{Introduction}
\noindent The main objective of this paper is to establish the existence and multiplicity of nontrivial solutions for the following problem.
\begin{align}\label{main}
\begin{split}
(-\Delta)^{\frac{1}{2}}u+u+(\ln|\cdot|*|u|^2)&=f(u)+\mu|u|^{-\gamma-1}u~\text{in}~\mathbb{R},
\end{split}
\end{align}
where $\mu>0$, $(*)$ is the {\it convolution} operation between two functions, $0<\gamma<1$. The fractional Laplacian operator $(-\Delta)^{\frac{1}{2}}$ is defined as
\begin{align}\label{laplacian}
\begin{split}
(-\Delta)^{\frac{1}{2}}u(x)&=C\left(1,\frac{1}{2}\right)\lim_{\epsilon\rightarrow 0}\int_{\mathbb{R}\setminus (x-\epsilon,x+\epsilon)}\frac{u(x)-u(y)}{|x-y|^{2}}dy,~\text{for every}~x\in\mathbb{R},
\end{split}
\end{align}
where $C\left(1,\frac{1}{2}\right)=\frac{2}{\pi}$ is a normalization constant (cf. 
{\sc Ros-Oton-Serra} \cite [Formula $A.1$]{rosoton1}).
Here, $f$ is a continuous function with an exponential growth whose primitive is $F(t)=\int_{0}^tf(s)ds$. Using subcritical or critical polynomial growth of the function $f$ is quite common in the literature pertaining to problems on elliptic PDEs. However, very few have considered the case when the nonlinear term has an exponential subcritical or a critical growth. To begin with, we first recall that a function $f$ is said to have a subcritical exponential growth at $\infty$ if
$$\underset{t\rightarrow\infty}{\lim}\frac{f(t)}{e^{\beta t^2}-1}=0,~\text{for every}~\beta>0.$$
Similarly, $f$ is said to be of critical exponential growth at $\infty$ if there exist $\theta\in(0,\pi]$ and $\beta_0\in(0,\theta)$ such that
\[\underset{t\rightarrow\infty}{\lim}\frac{f(t)}{e^{\beta t^2}-1}= \begin{cases}
0, & \text{for every}~\beta>\theta \\
\infty, & \text{for every}~\beta<\theta
\end{cases}\]
The following are some hypotheses which are commonly assumed for problems with the Moser-Trudinger inequality (cf.
 {\sc do \'{O} et al.} \cite{miya1}
  and
   {\sc Felmer et al.} \cite{felm1}):
\begin{eqnarray*}
(A_1)&~&f\in C(\mathbb{R},\mathbb{R}), f(0)=0,~\text{has critical exponential growth and}~F(t)\geq 0,~\text{for every}~t\in\mathbb{R};\\
(A_2)&~&\underset{|t|\rightarrow 0}{\lim}\frac{f(t)}{|t|}=0;\\
(A_3)&~&\text{there exists}~L>4~\text{such that}~f(t)t\geq L F(t)>0,~\text{for every}~t\in\mathbb{R}~\text{(this condition is used}\\
& &\text{ to verify that a {\it Cerami sequence} is bounded in the {\it Sobolev space}}~H^{\frac{1}{2}}(\mathbb{R}));\\
(A_4)&~&\text{there exist}~q>4~\text{and}~C_q>\frac{[2(q-2)]^{\frac{q-2}{2}}}{q^{\frac{q}{2}}}\frac{(S_q)^q}{r_0^{q-2}}~\text{such that}~F(t)\geq C_q|t|^q,~\text{for every}~t\in\mathbb{R},\\
& &~\text{where}~S_q, r_0>0~(S_q~\text{will be defined in Lemma}~\ref{bdd_seq}).
\end{eqnarray*}
\noindent We now give a short review of the related results. The problem that inspired us to investigate the current problem is from the paper by 
{\sc Bo\"{e}r-Miyagaki} \cite{miya0}.
 The novelty addressed in this work is due to the presence of a singular term which is difficult to handle since the corresponding energy functional ceases to be $C^1$. This poses an extra challenge in applying the Mountain pass theorem and other results in variational methods that demand the functional to be $C^1$. To add to these difficulties, the logarithmic term poses a great challenge to establishing the existence of a convergent subsequence of a Cerami sequence. The idea of module translations from the paper by 
{\sc Cingolani-Weth} \cite{13} 
will be used. However, since the norm of the space $X$ (the solution space which will be defined in the next section) is not invariant under translations, a new difficulty arises. The issues pertaining to the exponential term will be explained later.\\
Problems involving nonlocal operators are important in many fields of science and engineering e.g. optimization, finance, phase transitions, stratified materials crystal dislocations, anomalous materials, semipermeable membranes, flame propagation, water waves, soft thin films, conservation laws, etc. We refer the reader to
 {\sc Caffarelli}~\cite{10},
  {\sc Di Nezza et al.}~\cite{16}, 
  and the references therein. The literature pertaining to problems without the logarithmic Choquard and the singular term is quite vast and is impossible to list everything in this paper. A seminal work in the field of singularity driven problems is due to
   {\sc Lazer-Mc Kenna}~\cite{lazer}. 
   Thereafter the problems with singularity were studied by many researchers; see   
     {\sc Ghanmi-Saoudi}~\cite{saoudi1},
       {\sc Oliva-Petitta}~\cite{petitta1},
     {\sc Saoudi et al.}~\cite{ghosh1}, 
     and the references therein. 
     Some of the other works that the readers can consult involve the fractional Laplacian operator $(-\Delta)^s$ when $2s<N$ and $s\in(0,1)$ are   
            {\sc Chang-Wang} \cite{14} and
        {\sc Felmer et al.} \cite{felm1}.
        We note that 
         {\sc Felmer et al.}~\cite{felm1}
         also studied some properties of the solutions besides regularity. Furthermore,
          {\sc do \'{O} et al.} \cite{18}
           studied the problem without the singular and the logarithmic term with potentials that vanish at infinity. Some more suggested papers are 
           {\sc Autuori-Pucci}~\cite{3}, 
           {\sc Cao}~\cite{11},
 {\sc do \'{O} et al.}~\cite{20}, 
 {\sc Iannizzotto-Squassina}~\cite{26},
  {\sc Lam-Lu}~\cite{28},
  {\sc Moser} \cite{33},
  and
  {\sc Pucci et al.}~\cite{27}.\\ 
We now turn our attention to the problems involving Choquard logarithmic term. Some related references are
 {\sc Alves-Figueiredo} \cite{2}, 
 {\sc Cingolani-Jeanjean} \cite{12},
  {\sc Cingolani-Weth} \cite{13},
   {\sc Du-Weth} \cite{21}, 
   and
   {\sc Wen et al.} \cite{39}.
   Furthermore,
    {\sc Cingolani-Weth} \cite{13} 
    proved the existence of infinitely many distinct solutions and a ground-state solution, with $V:\mathbb{R}^2\rightarrow(0,\infty)$ which is continuous and $\mathbb{Z}^2$-periodic, $f(u)=b|u|^{p-2}u$, $b>0$. Since the setting is periodic, the global Palais-Smale condition can fail due to the invariance of the functional under the $\mathbb{Z}^2-$ translations. To tackle this problem,
     {\sc Du-Weth} \cite{21}
      proved the existence of a mountain pass solution and a ground state solution for the local problem \eqref{main} but without the singular term in the case when $V(x)\equiv \alpha>0$, $2<p<4$, and $f(u)=|u|^{p-2}u$. They went on to further prove that if $p\geq 3$, then both the energy levels are the same and provided a characterization for them. 
      {\sc Cingolani-Jeanjean} \cite{12}
       proved the existence of stationary waves with prescribed norm by considering $\lambda\in\mathbb{R}$.
        {\sc Wen et al.} \cite{39} 
        considered a nonlinearity with a polynomial growth. 
       {\sc Alves-Figueiredo} \cite{2}
        proved the existence of a ground state solution to problem \eqref{main} without the singular term  but with a nonlinearity of the Moser-Trudinger type. 
        For Choquard problems one can refer to
         {\sc Abdellaoui-Bentifour} \cite{choq1},           
          {\sc Biswas-Tiwari} \cite{biswas1}
           {\sc Bonheure et al.} \cite{9},{
           \sc Goel et al.} \cite{goel1},
            {\sc Guo-Wu} \cite{25},                  
                            {\sc Lieb} \cite{31},
              and
              {\sc Panda et al.} \cite{choq2}.\\
This paper is organized as follows. Section $2$ is a quick look at the mathematical background, space description. In Section $3$ we describe an application of the fractional Laplacian operator for dimension $N=1$ along with a few auxiliary lemmas.  In Section $4$ we prove a few auxiliary lemmas and our main result. Finally, in Section $5$, we give an appendix to the proofs of all results which have been used in the proof of the main theorem.
\section{Preliminaries}
\noindent This section is devoted to presentation of the most important notations, results, remarks that will be used in our study of problem \eqref{main}
(for the remaining background material we refer the reader to the comprehensive monograph by
{\sc Papageorgiou-R\u{a}dulescu-Repov\v{s}} \cite{PRR}), and the statement of the main result.\\
We begin by defining the Hilbert space
$$W^{\frac{1}{2},2}(\mathbb{R})=H^{\frac{1}{2}}(\mathbb{R})=\left\{u\in L^2(\mathbb{R}):\iint_{\mathbb{R}\times\mathbb{R}}\frac{|u(x)-u(y)|^2}{|x-y|^{2}}dxdy<\infty\right\},$$
equipped with the norm 
\begin{align}\label{norm1}
\begin{split}
\|u\|^2&=\|u\|_2^2+\iint_{\mathbb{R}\times\mathbb{R}}\frac{|u(x)-u(y)|^2}{|x-y|^{2}}dxdy=\|u\|_2^2+[u]_{\frac{1}{2},2}^2.
\end{split}
\end{align}
We denote the Schwartz class of functions by $\mathcal{S}(\mathbb{R})$.  Thus for any $u\in\mathcal{S}(\mathbb{R})$, the Fourier transform of $(-\Delta)^{\frac{1}{2}}u$ is given by $|\xi|\hat{u}$, where $\hat{u}$ denotes the Fourier transform of $u$. 
Also, by Proposition $3.6$ given in
 {\sc Di Nezza et al.} \cite{16}, 
 we have
$$\|(-\Delta)^{\frac{1}{4}}u\|_2^2=\frac{1}{2\pi}\iint_{\mathbb{R}\times\mathbb{R}}\frac{(u(x)-u(y))^2}{|x-y|^{2}}dxdy,~\text{for every}~u\in H^{\frac{1}{2}}(\mathbb{R}).$$
The factor $\frac{1}{2\pi}$ will be ignored in the paper throughout.   

\noindent Next, we define a slightly smaller space that will make the associated energy functional well-defined (cf.
 {\sc Stubbe}~\cite[Lemma 2.1]{38}):
$$X=\left\{u\in H^{\frac{1}{2}}(\mathbb{R}):\int_{\mathbb{R}}\ln(1+|x|)(u(x))^2dx<\infty\right\},$$
endowed with the norm
\begin{align}\label{norm2}
\begin{split}
\|u\|_X^2&=\|u\|_2^2+\int_{\mathbb{R}}\ln(1+|x|^2)(u(x))^2dx=\|u\|_2^2+\|u\|_{*}^2.
\end{split}
\end{align}
Then $X$ is a Hilbert space as well. We define three auxiliary bilinear forms as follows
\begin{align}\label{bilmap}
A(u,v)&=\iint_{\mathbb{R}\times\mathbb{R}}\ln(1+|x-y|)u(x)v(y)dxdy,\\
B(u,v)&=\iint_{\mathbb{R}\times\mathbb{R}}\ln\left(1+\frac{1}{|x-y|}\right)u(x)v(y)dxdy,\\
C(u,v)&=A(u,v)-B(u,v)=\iint_{\mathbb{R}\times\mathbb{R}}\ln(|x-y|)u(x)v(y)dxdy.
\end{align}
We further define the functionals
\begin{align}\label{fnalmap}
U:H^{\frac{1}{2}}(\mathbb{R})\rightarrow\mathbb{R},~V:L^4(\mathbb{R})\rightarrow\mathbb{R},~W:H^{\frac{1}{2}}(\mathbb{R})\rightarrow\mathbb{R},
\end{align}
by $U(u)=A(u^2,u^2)$, $V(u)=B(u^2,u^2)$, $W(u)=C(u^2,u^2)$.\\
Clearly, a combination of the Hardy-Littlewood-Sobolev inequality (HLS)
 {\sc Lieb} \cite{32}, 
 $0\leq\ln(1+r)\leq r$, for any $r>0$, leads to the inequality.
\begin{align}\label{cont_B}
\begin{split}
|B(u,v)|&\leq\iint_{\mathbb{R}\times\mathbb{R}}\frac{1}{|x-y|}u(x)v(y)dxdy\leq C_0\|u\|_2\|v\|_2,~\text{for every}~u,v\in L^2(\mathbb{R}),
\end{split}
\end{align}
where $C_0$ is an (HLS) constant. Consequently, we also have
\begin{align}\label{cont_V}
\begin{split}
|V(u)|&\leq C_0\|u\|_4^4,~\text{for every}~u\in L^4(\mathbb{R}). 
\end{split}
\end{align}
A standard property of the logarithmic function is
\begin{align}\label{log_fun_prop}
\ln(1+|x\pm y|)&\leq\ln(1+|x|+|y|)\leq\ln(1+|x|)+\ln(1+|y|),~\text{for every}~x,y\in\mathbb{R}.
\end{align}
Using, \eqref{log_fun_prop} in tandem with the H\"{o}lder inequality, we get
\begin{align}\label{cont_A}
A(uv,wz)&\leq\|u\|_{*}\|v\|_{*}\|w\|_2\|z\|_2+\|u\|_{2}\|v\|_{2}\|w\|_{*}\|z\|_{*},~\text{for every}~u,v,w,z\in L^2(\mathbb{R}).
\end{align}
\noindent The following lemma is the celebrated result of {\sc Moser-Trudinger} \cite{34}.
\begin{lemma}\label{MT}
	There exists $0<\omega\leq \pi$ such that for all $\beta\in(0,\omega)$, there exists a constant $C_{\beta}>0$ satisfying
	$$\int_{\mathbb{R}}(e^{\beta u^2}-1)dx\leq C_{\beta}\|u\|_2^2,$$
	for all $u\in H^{\frac{1}{2}}(\mathbb{R})$ with $\|(-\Delta)^{\frac{1}{4}}u\|\leq 1$.
\end{lemma}
\noindent The associated energy functional 
\begin{align}\label{energy}
E(u)&=\frac{1}{2}\|u\|^2+\frac{1}{4}W(u)-\int_{\mathbb{R}}F(u)dx-\frac{\mu}{1-\gamma}\int_{\mathbb{R}}|u|^{1-\gamma}dx.
\end{align}
is well-defined due to the lemma above and the space definition. However, the functional is not $C^1$ which disallows the use of the basic results of variational analysis. To tackle this, we define the cutoff functional $\tilde{E}$ as follows
\begin{align}\label{energy_mod}
\tilde{E}(u)&=\frac{1}{2}\|u\|^2+W(u)-\int_{\mathbb{R}}\tilde{G}(u)dx.
\end{align}
Here
\[   
\tilde{G}(t) = 
\begin{cases}
\mu |t|^{1-\gamma}+F(t), &~\text{if}~|t|>\underline{u}_{\mu}\\
\mu \underline{u}_{\mu}^{1-\gamma}+F(\underline{u}_{\mu}),&~\text{if}~|t|\leq \underline{u}_{\mu}
\end{cases}\]
where $F(t)=\int_0^tf(s)ds$ and $\underline{u}_{\mu}$ is a solution to
\begin{eqnarray}\label{auxprob}
(-\Delta)^{\frac{1}{2}}u+u+(\ln|\cdot|*|u|^2)&=\mu u^{-\gamma},~\text{in}~\mathbb{R},
\end{eqnarray}
whose existence is guaranteed by Lemma \ref{existence_positive_soln}. Moreover, by Lemma \ref{u_greater_u_lambda} in Appendix, one can establish that for a range of $\mu$, a solution to \eqref{main} is such that $u>\underline{u}_{\mu}$ a.e. in $\Omega$.\\
\noindent Furthermore, observe that $$\underset{u\rightarrow\infty}{\lim}\frac{\tilde{G}(u)}{e^{\beta u^2}-1}=0.$$
Also, we have $$0<\frac{\tilde{G}(u)}{e^{\beta  u^2}-1}<\frac{\tilde{G}(u)}{e^{\beta \underline{u}_{\lambda}^2}-1}.$$ 
Under the limit $u\rightarrow 0$, we have $\frac{\tilde{G}(u)}{e^{\beta \underline{u}_{\lambda}^2}-1}\rightarrow\frac{\underline{u}_{\lambda}^{1-\gamma}}{e^{\beta \underline{u}_{\lambda}^2}-1}$.
For $x\in\Omega$ which satisfy $0<\underline{u}_{\lambda}(x)<M<<1$, we have $(\underline{u}_{\lambda}(x))^2<(\underline{u}_{\lambda}(x))^{1-\gamma}$. Hence $\frac{1}{\beta}\approx\frac{\underline{u}_{\lambda}(x)^2}{e^{\beta  \underline{u}_{\lambda}(x)^2}-1}<\frac{\underline{u}_{\lambda}(x)^{1-\gamma}}{e^{\beta\underline{u}_{\lambda}(x)^2}-1}$. 
This implies that for a suitable $C'>0$ we have
$$\tilde{G}(u)\leq C' (e^{\beta u^2}-1)~\text{a.e. in}~\Omega.$$
By Lemmas  \ref{MT}, \ref{comp_emb} we can now conclude that 
\begin{align}\label{sing_est}
\int_{\mathbb{R}}\tilde{G}(u)dx \leq C' \int_{\mathbb{R}}(e^{\beta u^2}-1)dx\leq C_{\beta}\|u\|_2^2\leq C' \|u\|^2.
\end{align}
\begin{remark}\label{conseq0}
Under the hypothesis $(A_1)-(A_2)$, there exists, for $q>2$, $\epsilon>0$ and $\beta>\theta$, a constant $c_2>0$ such that 
\begin{align}\label{conseq1}
\begin{split}
|F(u)|&\leq \frac{\epsilon}{2}|u|^2+c_2|u|^{q}(e^{\beta|u|^2}-1),~\text{for every}~u\in X.
\end{split}
\end{align}
Furthermore, there exists a constant $c_3>0$ satisfying the following inequality
\begin{align}\label{conseq2}
\begin{split}
|f(u)|&\leq \epsilon|u|+c_3|u|^{q-1}(e^{\beta|u|^2}-1),~\text{for every}~u\in X.
\end{split}
\end{align}
An important consequence of \eqref{conseq1} is the following:\\
Consider $\rho_1, \rho_2>1$, $\rho_1\sim 1$, $\rho_2>2$, such that $\frac{1}{\rho_1}+\frac{1}{\rho_2}=1$.  Then it follows that for any $u\in H^{\frac{1}{2}}(\mathbb{R})$, $\epsilon>0$, $\beta>\theta$, we have that 
\begin{align}\label{expo_est}
\int_{\mathbb{R}}|F(u)|dx&\leq\frac{\epsilon}{2}\|u\|^2+c_2\|u\|_{\rho_2q}^q\left(\int_{\mathbb{R}}(e^{\rho_1\beta u^2}-1)dx\right)^{\frac{1}{\rho_1}}.
\end{align}
\end{remark}
\begin{remark}\label{subseq_rem}
Henceforth, \begin{enumerate}\item the notation of a subsequence will be the same as its sequence; \item the notation for cutoff energy functional $\tilde{E}$ will be continued to be denoted by $E$. \end{enumerate}
\end{remark} 

\noindent
We are now in a position to state our main result.
\begin{theorem}\label{main_res_1}
	Assume that hypotheses $(A_1)-(A_4)$ are satisfied and let $q>4$ and $C_q>0$ be chosen sufficiently large. Then
	\begin{enumerate}[label=(\roman*)]
		\item problem \eqref{main} has a solution $u\in X\setminus\{0\}$ such that 
		$$E(u)=\underset{\gamma\in\Gamma}{\inf}\underset{t\in[0,1]}\max\{E(\gamma(t))\}=d,$$
		where $$\Gamma=\{\gamma\in C([0,1],X):\gamma(0), E(\gamma(1))<0\}$$ is the class of paths on $X$ joining $\gamma(0)$ and $\gamma(1)$;
		\item problem \eqref{main} has a ground state solution $u\in X\setminus\{0\}$ such that\\ $$E(u)=\inf\{E(v):v\in X~\text{is a solution of problem} ~\eqref{main}\}.$$
	\end{enumerate}
\end{theorem} 

\noindent An important application of the fractional Laplacian operator for dimension $N=1$ can be found in 
{\sc Dipierro et al.}\cite{dipi}.  
We give a gist of the version of a model for the dynamics of the dislocation of atoms in crystals. The model is related to the Peierls-Nabarro energy functional. The system is a hybrid combination in which a discrete dislocation occurring along dislocation dynamics in crystals a slide line is incorporated in a continuum medium. The problem is as follows
\begin{align}\label{app_1}
\begin{split}
\frac{\partial}{\partial t}v&=(-\Delta)^sv-P'(v)+\sigma_{\epsilon}(t,x),~\text{in}~(0,+\infty)\times\mathbb{R},
\end{split}
\end{align}
where $s\in\left[\frac{1}{2},1\right)$, $P$ is a $1$-periodic potential and $\sigma_{\epsilon}$ plays the role of external stress acting on the material. Setting $v_{\epsilon}(t,x)=v\left(\frac{t}{\epsilon^{1+2s}},\frac{x}{\epsilon}\right)$, equation \eqref{app_1} can be recast as follows
\begin{align}\label{app_2}
\begin{split}
\frac{\partial}{\partial t}v_{\epsilon}&=\frac{1}{\epsilon}\left((-\Delta)^sv_{\epsilon}-\frac{1}{\epsilon^{2s}}P'(v_{\epsilon})+\sigma_{\epsilon}(t,x)\right),~\text{in}~(0,+\infty)\times\mathbb{R}\\
v_{\epsilon}(0,x)&=v_{\epsilon}^0,~\text{in}~\mathbb{R}.
\end{split}
\end{align}
For a suitable choice of $v_{\epsilon}^0$ the {\it basic layer solution} $u$ is introduced, which happens to be a solution to the following problem
\begin{align}\label{app_3}
\begin{split}
(-\Delta)^su-P'(u)=0,~\text{in}~\mathbb{R}\\
u'>0~\text{and}~u(-\infty)=0, u(0)=\frac{1}{2}, u(+\infty)=1.
\end{split}
\end{align}
One can see that the problem considered in this article has a proper physical application and is a testimony to the importance of the problem considered in this paper.
\section{Auxiliary Lemmas} 
\noindent In this section we will discuss some auxiliary lemmas. 
\begin{lemma}\label{comp_emb}(cf.
 {\sc Bo\"{e}r-Miyagaki} \cite[Lemma $2.1$]{miya0})
The space $X$ is continuously embedded in $H^{\frac{1}{2}}(\mathbb{R})$ and compactly embedded in $L^p(\mathbb{R})$, for every $p\geq 2$.
\end{lemma}
\noindent Below is the well-known Moser-Trudinger lemma (cf.
 {\sc Cao}~\cite{11}).
\begin{lemma}\label{MS_lemma}(cf. 
{\sc Ozawa} \cite{34})
Let $(u_n)$ be a sequence in $L^2(\mathbb{R})$ and $u\in L^2(\mathbb{R})\setminus\{0\}$ such that $u_n\rightarrow u$ pointwise a.e. on $\mathbb{R}$. Moreover, let $(v_n)$ be a bounded sequence in $L^2(\mathbb{R})$ such that \\
$\underset{n\in\mathbb{N}}{\sup} \{A(u_n^2,v_n^2)\}<\infty.$
Then there exist $n_0\in\mathbb{N}$ and $C>0$ such that $\|u_n\|_{*}<C$, for any $n\geq n_0$. Moreover, if
$$A(u_n^2,v_n^2)\rightarrow 0~\text{and}~\|v_n\|_2\rightarrow 0,~\text{as}~n\rightarrow\infty,$$
then $\|v_n\|_{*}\rightarrow 0$ as $n\rightarrow\infty$.
\end{lemma}
Based on Lemma \ref{MT},
 {\sc do \'{O} et al.}~\cite{miya1}
  proved the following lemma.
\begin{lemma}\label{miya_squassina}(cf.
 {\sc do \'{O} et al.}~\cite[Proposition $2.1$]{miya1})
For any $\beta>0$, $u\in W^{\frac{1}{2}}(\mathbb{R})$,\\
$$\int_{\mathbb{R}}(e^{\beta u^2}-1)dx<\infty.$$
\end{lemma}
Hence, from Lemma \ref{miya_squassina} and equation \eqref{conseq1}, for any $u\in X$ we have
\begin{align}\label{expo_est_1}
\int_{\mathbb{R}}|F(u)|&\leq\frac{\epsilon}{2}\|u\|^2+c_2\|u\|_{\rho_2q}^q\left(\int_{\mathbb{R}}(e^{\rho_1\beta u^2}-1)dx\right)^{\frac{1}{\rho_1}}<\infty.
\end{align}
The following are some useful lemmas which will be used in the paper.
\begin{lemma}\label{useful_lem_1}(cf. 
{\sc Cingolani-Weth} \cite[Lemma $2.6$]{13})
Let $(u_n)$, $(v_n)$, and $(w_n)$ be bounded sequences in $X$ such that $u_n\rightharpoonup u$ in $X$. Then for every $z\in X$, we have $A(v_n w_n,z\cdot(u_n-u))\rightarrow 0$ as $n\rightarrow\infty$.
\end{lemma}
\begin{lemma}\label{useful_lem_2}(cf.
 {\sc Felmer et al.} \cite[Lemma $2.2$]{felm1})
\begin{enumerate}[label=(\roman*)]
\item The functionals $U, V, W$ are of class $C^1$ on $X$. In fact, $\langle U'(u),v\rangle=4A(u^2,uv)$, $\langle V'(u),v\rangle=4B(u^2,uv)$, $\langle W'(u),v\rangle=4C(u^2,uv)$, for all $u,v\in X$.
\item $V$ is continuously differentiable on $L^4(\mathbb{R})$.
\item $U$ is a weakly lower semicontinuous functional on $H^{\frac{1}{2}}(\mathbb{R})$.
\item $E$ is lower semicontinuous on $H^{\frac{1}{2}}(\mathbb{R})$. 
\end{enumerate}
\end{lemma} 

\noindent
We now check that the energy functional $E$ depicts the Mountain pass geometry. This will be required to obtain a {\it Cerami sequence} for a certain mountain pass energy level $d$ (given in Theorem \ref{main_res_1}). Following is the definition of a {\it Cerami sequence} pertaining to a $C^1$-functional.
\begin{definition}\label{cerami_seq}
({\sc Soni-Choudhuri} \cite[Definition $2.1$]{soni})
Let $\Phi:Y\rightarrow\mathbb{R}$ be a $C^1$-functional, where $Y$ is a normed linear space with the norm $\|\cdot\|_Y$. Then $\Phi$ is said to satisfy the {\it Cerami} condition at a level $c\in\mathbb{R}$ if any sequence $(u_n)\subset Y$ such that $\Phi(u_n)\rightarrow c$ and $(1+\|u_n\|_Y)\Phi'(u_n)\rightarrow 0$ as $n\rightarrow\infty$ has a convergent subsequence in $Y$.
\end{definition}
\begin{remark}\label{ps_importance}\noindent The advantage of this condition is that a Cerami sequence can produce a critical point even when a Palais-Smale sequence does not (cf.
 {\sc Schechter} \cite{martin}).
\end{remark}
\begin{lemma}\label{mp_geo}
There exist sufficiently small $R>0$ and $\mu_0>0$ such that 
$$m_r=\inf\{E(u):u\in X, \|u\|=r\}>0,~\text{for every}~r\in (0,R]$$ and $$m'_r=\inf\{\langle E'(u),u\rangle:u\in X, \|u\|=r\}>0,~\text{for every}~r\in (0,R]$$ whenever $\mu\in(0,\mu_0)$.
\end{lemma}
\begin{proof}
Let $u\in X\setminus\{0\}$ and such that $u>\underline{u}_{\lambda}$ a.e. in $\Omega$. Choose $\rho_1, \rho_2>1$ such that $\rho_1\sim 1$, $\rho_2>2$ and
$\rho_1\beta\|u\|^2$ in order to apply the exponential estimates in \eqref{expo_est}. Furthermore, on using the Sobolev embeddings, we get
\begin{align}\label{est1}
\begin{split}
E(u)&=\frac{1}{2}\|u\|^2+\frac{1}{4}W(u)-\int_{\mathbb{R}}F(u)dx-\frac{\mu}{1-\gamma}\int_{\mathbb{R}}|u|^{1-\gamma}dx\\
&\geq \frac{1}{2}\|u\|^{2}\left((1-\epsilon)-C_2\|u\|^{2}-C_3\|u\|^{q-2}-\frac{C_4\mu}{1-\gamma}\right).
\end{split}
\end{align}
Thus, for a pair of sufficiently small positive numbers $(\mu_0,R)$, we get $E(u)>0$, whenever $\|u\|=r<R$ and $\mu\in(0,\mu_0)$. Similarly, 
\begin{align}\label{est2}
\begin{split}
\langle E'(u),u\rangle&=\|u\|^2+W(u)-\int_{\mathbb{R}}f(u)udx-\mu\int_{\mathbb{R}}|u|^{1-\gamma}dx\\
&\geq \|u\|^{2}\left((1-\epsilon)-C_4\|u\|^{2}-C_5\|u\|^{q-2}-C_4\mu\right).
\end{split}
\end{align}
Once again, by choosing $R>0$ and $\mu_0$ sufficiently small we obtain $m_r'>0$. This completes the proof.
\end{proof}
\begin{lemma}\label{useful_lemma1}
Let $u\in X\setminus\{0\}$ and $q>4$. Then 
$$\underset{t\rightarrow 0}{\lim}{E(tu)}=0,~\underset{t>0}{\sup}\{E(tu)\}<\infty,~E(tu)\rightarrow-\infty~\text{as}~t\rightarrow\infty.$$
\end{lemma}
\begin{proof}
Suppose that $u\in X\setminus\{0\}$. By $(A_4)$ and $q>4$ we get
\begin{align}\label{est3}
\begin{split}
E(tu)&=\frac{t^2}{2}\|u\|^2+\frac{t^4}{4}W(u)-\int_{\mathbb{R}}F(tu)dx-\frac{t^{1-\gamma}\mu}{1-\gamma}\int_{\mathbb{R}}|u|^{1-\gamma}dx\\
&\leq\frac{t^2}{2}\|u\|^2+\frac{t^4}{4}W(u)-C_qt^q\int_{\mathbb{R}}|u|^qdx-\frac{t^{1-\gamma}\mu}{1-\gamma}\int_{\mathbb{R}}|u|^{1-\gamma}dx\\ &\leq\frac{t^2}{2}\|u\|^2+\frac{t^4}{4}W(u)-C_qt^q\|u\|_q^q-\frac{t^{1-\gamma}\mu}{1-\gamma}\int_{\mathbb{R}}|u|^{1-\gamma}dx\rightarrow-\infty~\text{as}~t\rightarrow\infty.
\end{split}
\end{align}
Also, it is easy to see that $\underset{t\rightarrow 0}{\lim}{E(tu)}=0,~\underset{t>0}{\sup}\{E(tu)\}<\infty$.
\end{proof}
\begin{remark}\label{cerami_rem}
It is easy to verify invoking Lemma \ref{useful_lemma1} and the intermediate value theorem that $0<m_R\leq d<\infty$. Since, the functional $E$ has mountain pass geometry, by
 {\sc Du-Weth} \cite[Lemma $3.2$]{21}, 
 there exists a sequence $(u_n)\subset X$ such that \begin{eqnarray}\label{cerami}E(u_n)\rightarrow d~\text{and}~ \|E'(u_n)\|_{X'}(1+\|u_n\|_X)\rightarrow 0,~\text{as}~n\rightarrow\infty.\end{eqnarray}
\end{remark}
\begin{lemma}\label{Cerami_bdd}
Suppose that the sequence $(u_n)\subset X$ satisfies the Cerami condition \eqref{cerami}. Then $(u_n)$ is bounded in $H^{\frac{1}{2}}(\mathbb{R})$.
\end{lemma}
\begin{proof}
Using \eqref{sing_est}, the condition in \eqref{cerami}, $(A_3)$, and the embedding of $H^{\frac{1}{2}}(\mathbb{R})\hookrightarrow L^q(\mathbb{R})$ for $q\in[1,\infty)$ from Lemma \ref{comp_emb}, we obtain 
\begin{align}\label{est4}
\begin{split}
d+o(1)&\geq E(u_n)-\frac{1}{4}\langle E'(u_n),u_n \rangle\\
&\geq \frac{1}{4}\|u_n\|^2+\left(\frac{L}{4}-1\right)\int_{\mathbb{R}}F(u_n)dx-\mu\frac{3+\gamma}{4(1-\gamma)}\int_{\mathbb{R}}|u|^{1-\gamma}dx\\
&\geq\frac{1}{4}\|u_n\|^2+C_6\left(\frac{L}{4}-1\right)\|u_n\|^q-\mu\frac{3+\gamma}{4(1-\gamma)}\int_{\mathbb{R}}|u|^{1-\gamma}dx\\
&\geq\frac{1}{4}\|u_n\|^2+C_6\left(\frac{L}{4}-1\right)\|u_n\|^q-C'\mu\frac{3+\gamma}{4(1-\gamma)}\|u_n\|^{2}.
\end{split}
\end{align}
The inequality in \eqref{est4} clearly shows that the sequence $(u_n)$ is bounded in $H^{\frac{1}{2}}(\mathbb{R})$. For if not, then on dividing \eqref{est4} by $\|u_n\|^q$ and then passing the limit $n\rightarrow\infty$, yields a contradiction to $0\geq C_6\left(\frac{L}{4}-1\right)$. Thus, for a small range of $\mu$, say $(0,\mu_0)$, we have
\begin{align}\label{est4'}
\begin{split}
d+o(1)\geq\frac{1}{4}\|u_n\|^2.
\end{split}
\end{align}
\end{proof}
\noindent The following lemma shows that any sequence $(u_n)\subset X$ such that $E(u_n)\leq d$ for all $n\in\mathbb{N}$ can be taken to be of sufficiently small norms.
\begin{lemma}\label{bdd_seq}
Let $(u_n)\subset X$ satisfy the Cerami condition in \eqref{cerami} with $q>4$. Then for some sufficiently small $r_0>0$, we have $\underset{n}{\lim\sup}\|u_n\|^2<r_0^2$.
\end{lemma}
\begin{proof}
By Lemma \ref{Cerami_bdd}, we know that $(u_n)$ is bounded in $H^{\frac{1}{2}}(\mathbb{R})$. Certainly,  $\underset{n}{\lim\sup}\|u_n\|^2\leq 4c+o(1)$ is bounded above (and of course, below). We will find an estimate of the upper bound for this quantity. Consider the set $\mathfrak{S}=\{u\in X:u\neq 0, W(u)\leq 0\}$ and define $u_t(x)=t^2u(tx)$ for all $t>0$, $u\neq 0\in X$, $x\in\mathbb{R}$. We have 
$W(u_t)=t^6C(U)-t^6\ln t\|u\|_2^4\rightarrow-\infty$ as $t\rightarrow\infty$. This shows that $\mathfrak{S}$ is nonempty. By the Sobolev embedding theorem we have $\|u\|\geq C\|u\|_q$, for all $u\in H^{\frac{1}{2}}(\mathbb{R})\setminus\{0\}$.\\
We define $$S_q(v)=\frac{\|v\|}{\|v\|_q}.$$
Therefore, $S_q=\inf_{v\in\mathfrak{S}}S_q(v)\geq\inf_{v\neq 0}S_q(v)>0.$ We will now estimate the energy level $d$. Let $v\in\mathfrak{S}$, and $T>0$ be sufficiently small. Then $E(Tv)<0$. Take a path $\alpha:[0,1]\rightarrow X$ defined as $\alpha(t)=tTv$. Therefore,
\begin{align}\label{comp1}
d\leq\underset{0\leq t \leq 1}{\max}E(\alpha(t))=\underset{0\leq t \leq 1}{\max}E(tTv)\leq\underset{t\geq 0}{\max}E(tv).
\end{align}
Consequently, for $w\in\mathfrak{S}$, we have 
\begin{align}\label{comp2}
d\leq\underset{t\geq 0}{\max}E(tw)\leq\underset{t\geq 0}{\max}\left\{\frac{t^2}{2}\|w\|^2-C_qt^q\|w\|_q^q\right\}\leq \left(\frac{1}{2}-\frac{1}{q}\right)\frac{(S_q(w))^{\frac{2q}{q-2}}}{(qC_q)^{\frac{2}{q-2}}},
\end{align}
where we have used $(A_4)$. On taking infimum over $w\in\mathfrak{S}$, we obtain 
\begin{align}\label{comp3}
\underset{n}{\lim\sup}\|u_n\|^2\leq 4d\leq 2\cdot\frac{q-2}{q}\frac{(S_q(w))^{\frac{2q}{q-2}}}{(qC_q)^{\frac{2}{q-2}}}\leq r_0^2.
\end{align}
\end{proof}
Before we state and prove the next lemma we need to recall the following two theorems.
\begin{theorem}\label{hitch}
({\sc Di Nezza et al.} \cite[Theorem $7.1$]{16})
Let $s\in(0,1)$, $p\in[1,\infty)$, $q\in [1,p]$, $\Omega\subset\mathbb{R}^N$ be a bounded extension domain for $W^{s,p}$ and $T$ be a bounded subset of $L^p(\Omega)$. Suppose
that $$\underset{f\in T}{\sup}\iint_{\Omega\times\Omega}\frac{|f(x)-f(y)|^p}{|x-y|^{N+sp}}dxdy<\infty.$$
Then $T$ is pre-compact in $L^q(\Omega)$.
\end{theorem}
\begin{theorem}\label{adam}
({\sc Adams} \cite[Theorem $7.41$]{adam})
Suppose that for $\Omega\subset\mathbb{R}^N$ there exists a strong $(M+1)-$extension operator $\mathfrak{O}$ and, for $|\delta|\leq|\alpha|=M$, linear operators $\mathfrak{O}_{\alpha\delta}$ continuous from $W^{s,p}(\Omega)$ into $W^{1,p}(\mathbb{R}^N)$ and from $L^p(\Omega)$ into $L^{p}(\mathbb{R}^N)$ such that if $u\in W^{m,p}(\Omega)$, then
\begin{align}\label{est}
D^{\alpha}\mathfrak{O}u(x)&=\underset{|\delta|\leq M}{\sum}\mathfrak{O}_{\alpha\delta}D^{\delta}
u(x).
\end{align}
If $s=M+\sigma>0$, $0\leq\sigma<1$, then $W^{s,p}(\Omega)$ coincides with the set of restrictions on $\Omega$ of functions in $W^{s,p}(\mathbb{R}^N)$.
\end{theorem}
\noindent Finally, we have the following lemma.
\begin{lemma}\label{conc_comp}
Let $(u_n)\subset X$ be bounded in $H^{\frac{1}{2}}(\mathbb{R})$ and such that 
\begin{align}\label{conc_comp_1}
Q=\underset{n}{\lim\inf}\underset{y\in\mathbb{Z}}{\sup}\int_{B_{\frac{3}{2}}(y)}|u_n(x)|^2dx>0.
\end{align}
Then there exist $u\in H^{\frac{1}{2}}(\mathbb{R})\setminus\{0\}$ and $(y_n)\subset\mathbb{Z}$ such that up to a subsequence, $\tilde{u}_n=u_n(\cdot-y_n)\rightharpoonup u\in H^{\frac{1}{2}}(\mathbb{R})$. Here,  $B_{\frac{3}{2}}(y)=\{x\in\mathbb{R}:|x-y|<\frac{3}{2}\}$.
\end{lemma}
\begin{proof}
The property of $\lim\inf$ and $\sup$ together produces the sequence $(y_n)$ such that $|y_n|\rightarrow\infty$ and the boundedness of $(u_n)$ in $H^{\frac{1}{2}}(\mathbb{R})$ produces  $u$ such that $u_n\rightharpoonup u\in H^{\frac{1}{2}}(\mathbb{R})$. Also, $u\neq 0$ given condition \eqref{conc_comp_1}. By Lemma \ref{comp_emb} we have that $u_n\rightarrow u$ in $L^2(\mathbb{R})$ and hence from Theorem \ref{hitch} we have $u_n\rightarrow u$   in $L^2(B_{\frac{3}{2}}(y_n))$. Therefore, there exists a subsequence such that $u_n\rightarrow u$ a.e. in $B_{\frac{3}{2}}(y_n)$. These $\tilde{u}_n=u_n(\cdot-y_n)$ are nothing but the restrictions of the functions in $(u_n)$ over $B_{\frac{3}{2}}(y_n)$. Thus by Theorem \ref{adam}, we can consider the functions $\mathfrak{O}_{00}\tilde{u}_n$ that will still be denoted by $\tilde{u}_n$. Therefore, $\tilde{u}_n\rightharpoonup u\in H^{\frac{1}{2}}(\mathbb{R})$.
\end{proof}
\section{Proof of the Main Theorem}
\noindent We now give a proof of Theorem \ref{main_res_1}. First, we will prove another lemma.
\begin{lemma}\label{alternative}
Let $q>2$ and $(u_n)\subset X$ satisfying the Cerami condition. On passing to a subsequence, if necessary, exactly one of the following statements holds true.
\begin{enumerate}[label=(\roman*)]
\item $\|u_n\|\rightarrow 0$ and $E(u_n)\rightarrow 0$ as $n\rightarrow\infty$.
\item There exists $(y_n)\subset\mathbb{Z}$ such that $|y_n|\rightarrow\infty$ such that $\tilde{u}_n=u_n(\cdot-y_n)\rightarrow u$ in $X$, for a nontrivial critical point $u\in X$ of $E$.
\end{enumerate}	
\end{lemma}
\begin{proof}
By Lemma \ref{bdd_seq}, we see that the Cerami sequence $(u_n)$ is bounded in $H^{\frac{1}{2}}(\mathbb{R})$. Therefore there exists a subsequence from Lemma \ref{conc_comp} such that 
\begin{align}\label{ineq1}
\int_{\mathbb{R}}(e^{\rho_1\beta u_n^2}-1)dx&\leq C_{\beta},~\text{for every}~n\in\mathbb{N}.
\end{align}
Suppose that $(i)$ does not hold.\\
{\it Claim $A$}:~$\underset{n}{\lim\inf}\underset{y\in\mathbb{Z}}{\sup}\int_{B_2(y)}|u_n(x)|^2dx>0$.\\
On the contrary, suppose this does not hold. By Lion's Lemma $2.4$ in
 {\sc Yu et al.} \cite{40}, 
 we have $u_n\rightarrow 0$ in $L^{a}(\mathbb{R})$ for all $a>2$. By \eqref{cont_V}, we have $V(u_n)\rightarrow 0$ as $n\rightarrow\infty$. Since $q>2$ we have by Remark \ref{conseq0}, for a subsequence, that    
\begin{align}\label{ineq2}
\begin{split}
\left|\int_{\mathbb{R}}f(u_n)u_n dx\right|\leq\epsilon\|u_n\|_2^2+C_7\|u_n\|_{qr_2}^q\rightarrow 0,
\int_{\mathbb{R}}|u_n|^{1-\gamma}dx\rightarrow 0~\text{as}~\epsilon\rightarrow 0, n\rightarrow\infty. 
\end{split}
\end{align}
Furthermore, 
\begin{align}\label{ineq3}
\begin{split}
\|u_n\|^2+U(u_n)&=\langle E'(u_n),u_n\rangle+V(u_n)+\int_{\mathbb{R}}f(u_n)u_ndx+\mu\int_{\mathbb{R}}|u_n|^{1-\gamma}dx\rightarrow 0~\text{as}~n\rightarrow\infty. 
\end{split}
\end{align}
This further implies that $\|u_n\|, U(u_n)\rightarrow 0$ as $n\rightarrow\infty$. By the embedding $H^{\frac{1}{2}}(\mathbb{R})\hookrightarrow L^2(\mathbb{R})$, we have that $\|u_n\|_2\rightarrow 0$ as $n\rightarrow\infty$. Also by Remark \ref{conseq0}, we have $\int_{\mathbb{R}}F(u_n)dx\rightarrow 0$ as $n\rightarrow\infty$. Thus $E(u_n)\rightarrow 0$ as $n\rightarrow\infty$ which is a contradiction. Thus $$\underset{n}{\lim\inf}\underset{y\in\mathbb{Z}}{\sup}\int_{B_2(y)}|u_n(x)|^2dx>0.$$
This fact combined with Lemma \ref{conc_comp}, helps us to produce a sequence $(y_n)\subset\mathbb{Z}$ and $u\in H^{\frac{1}{2}}(\mathbb{R})\setminus\{0\}$ such that $u_n(\cdot-y_n)=\tilde{u}_n\rightharpoonup u$ in $H^{\frac{1}{2}}(\mathbb{R})$. Therefore $(\tilde{u}_n)$ is bounded in $L^p(\mathbb{R})$ for any $p\geq 2$ and hence $\tilde{u}_n(x)\rightarrow u(x)$ a.e. in $\mathbb{R}$. Now, observe that for $q>2$,
\begin{align}\label{ineq4}
\begin{split}
U(\tilde{u}_n)=U(u_n)&=\langle E'(u_n),u_n \rangle+V(u_n)+\int_{\mathbb{R}}f(u_n)u_ndx+\mu\int_{\mathbb{R}}|u_n|^{1-\gamma}dx-\|u_n\|^2\\
&\leq C_0\|u_n\|_4^2+\epsilon C_8+C_9\|u_n\|_{qr_2}^q+o(1)\leq C_{10}+o(1).
\end{split}
\end{align}
Therefore, $U(\tilde{u}_n)<\infty$. Hence, since $(\tilde{u}_n)$ is bounded in $L^2(\mathbb{R})$ it follows by Lemma \ref{MS_lemma} that, $(\|\tilde{u}_n\|_{*})$ is bounded in $X$. By the reflexivity of $X$, we get $\tilde{u}_n\rightharpoonup u$ in $X$. From Lemma \ref{comp_emb} we obtain $\tilde{u}_n\rightarrow u$ in $L^p(\mathbb{R})$ for all $p\geq 2$.\\
{\it Claim B:}~$\langle E'(\tilde{u}_n),\tilde{u}_n-u\rangle\rightarrow 0$ as $n\rightarrow\infty$.\\
A simple change of variable guarantees that $\langle E'(\tilde{u}_n), \tilde{u}_n-u\rangle=\langle E'(u_n),u_n-u(x+y_n)\rangle$. Thus,
\begin{align}\label{ineq5}
\begin{split}
|\langle E'(\tilde{u}_n),\tilde{u}_n-u\rangle|&=|\langle E'(u_n),u_n-u(x+y_n)\rangle|\\
&\leq\|E'(u_n)\|_{X'}(\|u_n\|_X+\|u(\cdot+y_n)\|_X),~\text{for every}~n\in\mathbb{N}.
\end{split}
\end{align}
We will now try to estimate $\|u(\cdot+y_n)\|_X$. Suppose that $|y_n|\rightarrow\infty$. Then for each $x\in\mathbb{R}$,
\begin{align}\label{ineq6}
\ln(1+|y_n|)-\ln(1+|x-y_n|)&=\ln\left(\frac{1+|y_n|}{1+|x-y_n|}\right)\rightarrow 0,~n\rightarrow\infty.
\end{align}
Thus there exists $C_{11}>0$ such that $\ln(1+|x-y_n|)\geq C_{11}\ln(1+|y_n|)$ for all $n\in\mathbb{N}$, which implies 
\begin{align}\label{ineq7'}
\begin{split}
\|u_n\|_*^2&\geq \int_{\mathbb{R}}\ln(1+|x-y_n|)\tilde{u}_n^2dx\\
&\geq C_{11}\ln(1+|y_n|)\int_{\mathbb{R}}\tilde{u}_n^2dx\geq C_{11}\ln(1+|y_n|),~\text{for every}~n\in\mathbb{N}.
\end{split}
\end{align}
Now, if $y_n\rightarrow y_0$, then obviously, up to a subsequence still denoted by $(y_n)$, we have $y_n=y_0$. Let $y_0>0$. Suppose that $\delta>0$ and define $S=(-b,0)$. For each $x\in S$, we have  $|x-y_0|>|y_0|$. Thus, by the Mean value theorem of integration, there exists $x_{b}\in S$ such that
\begin{align}\label{ineq7}
\begin{split}
\|\tilde{u}_n\|_{*}^2&\geq\int_{S}\ln(1+|x-y_n|)\tilde{u}_n^2dx=a\ln(1+|x-y_n|)\tilde{u}_n^2(x_a)\\
&=C_{11}'\ln(1+|x-y_n|)\geq C_{11}'\ln(1+|y_0|)=C_{11}'\ln(1+|y_n|),~\text{for some}~C_{11}'>0. 
\end{split}
\end{align}
Similarly, for the case of $y_0<0$ we arrive at a similar conclusion. Furthermore, when $y_0=0$, we have $\tilde{u}_n=u_n$ and hence the conclusion is straightforward. Thus for any of the cases there exists $C_{11}'>0$  (need not be the same constant but is denoted by the same name) such that 
\begin{align}\label{ineq8'}\|u_n\|_*^2\geq C_{11}'\ln(1+|y_n|),~\text{for every}~n\in\mathbb{N}.\end{align}
From \eqref{log_fun_prop} we obtain
\begin{align}\label{ineq8}
\begin{split}
\|\tilde{u}_n\|_*^2&=\int_{\mathbb{R}}\ln(1+|x+y_n|)u_n^2dx\leq\|u_n\|_*^2+\ln(1+|y_n|)\|u_n\|_2^2.
\end{split}
\end{align}
By the weak lower semicontinuity of norm and by the $\mathbb{Z}$-invariance of $\|\cdot\|_2$, it follows from \eqref{log_fun_prop}, \eqref{ineq7}, \eqref{ineq8} that
\begin{align}\label{ineq9}
\begin{split}
\|u(\cdot+y_n)\|_*^2&\leq \|u\|_*^2+\ln(1+|y_n|)\|u\|_2^2\leq  \|\tilde{u}_n\|_*^2+\ln(1+|y_n|)\|u_n\|_2^2\\
&=\|u_n\|^2+\|u_n\|_*^2(1+C_{12}\|u_n\|_2^2)\leq\|u_n\|^2+C_{13}\|u_n\|_*^2\leq C_{14}\|u_n\|_X^2
\end{split}
\end{align}
for $n\in\mathbb{N}$. This implies that there exists $C_{15}>0$ such that, on passing to a subsequence, we obtain,
\begin{align}\label{ineq10}
\begin{split}
\|u(\cdot+y_n)\|_*^2&=\|u\|^2+\|u(\cdot+y_n)\|_*^2\leq\|u_n\|^2+C_{14}\|u_n\|_X^2\leq C_{15}\|u_n\|_X^2.
\end{split}
\end{align}
By the Cerami condition we have 
\begin{align}\label{ineq11}
\begin{split}
|\langle E'(\tilde{u}_n),\tilde{u}_n-u\rangle|&\leq (1+\sqrt{C_{15}})\|E'(u_n)\|_{X'}\|u_n\|_X\rightarrow 0,~\text{as}~n\rightarrow\infty.\\
\end{split}
\end{align}
{\it Claim C:}~$\int_{\mathbb{R}}f(\tilde{u}_n)(\tilde{u}_n-u)dx\rightarrow 0$, as $n\rightarrow\infty$.\\
By the $\mathbb{Z}$-invariance of $\|\cdot\|$, the Moser-Trudinger inequality, and \eqref{expo_est} we get 
\begin{align}\label{ineq12}
\begin{split}
\left|\int_{\mathbb{R}}f(\tilde{u}_n)(\tilde{u}_n-u)dx\right|\leq\|\tilde{u}_n\|_2\|\tilde{u}_n-u\|_2+C_{16}\|\tilde{u}_n-u\|_{qr_2}^q\rightarrow 0,~\text{as}~n\rightarrow\infty.
\end{split}
\end{align}
Hence Claim C has been verified.\\
Furthermore, 
\begin{align}\label{ineq13}
|\langle V'(\tilde{u}_n),\tilde{u}_n-u\rangle|\leq c_0\|\tilde{u}_n\|_4^3\|\tilde{u}_n-u\|_4\rightarrow 0
\end{align}
and
\begin{align}\label{ineq14}
\begin{split}
\langle U'(\tilde{u}_n),\tilde{u}_n-u\rangle&=A(\tilde{u}_n^2,\tilde{u}_n(\tilde{u}_n-u)^2)=A(\tilde{u}_n^2,(\tilde{u}_n-u)^2)+A(\tilde{u}_n^2,u(\tilde{u}_n-u)^2).
\end{split}
\end{align}
Since $(\tilde{u}_n)$ is bounded in $X$, invoking Lemma \ref{useful_lem_1}, we obtain $A(\tilde{u}_n^2,u(\tilde{u}_n-u)^2)\rightarrow 0$, as $n\rightarrow\infty$. Thus
\begin{align}\label{ineq15}
\begin{split}
o(1)&=\|\tilde{u}_n\|^2-\|u\|^2+A(\tilde{u}_n^2,(\tilde{u}_n-u)^2)+o(1)\geq \|\tilde{u}_n\|^2-\|u\|^2+o(1).
\end{split}
\end{align}
Hence $\|\tilde{u}_n\|\rightarrow\|u\|$ and $A(\tilde{u}_n^2,(\tilde{u}_n-u)^2)\rightarrow 0$. Therefore $\|\tilde{u}_n-u\|\rightarrow 0$ as $n\rightarrow\infty$. Furthermore, by Lemma \ref{MS_lemma}, we have $\|\tilde{u}_n-u\|_*\rightarrow 0$ which implies that $\tilde{u}_n\rightarrow u$ in $X$ as $n\rightarrow\infty$. It remains to show that $u$ is indeed a critical point of the functional $E$. As deduced in \eqref{ineq9}, one can show that there exists $C_{17}>0$ such that 
\begin{align}\label{ineq16}
\|v(\cdot+y_n)\|_X&\leq C_{17}\|u_n\|_X,~\text{for every}~n\in\mathbb{N}.
\end{align}
Therefore,
\begin{align}\label{ineq17}
\begin{split}
|\langle E'(u),v\rangle|=\underset{n\rightarrow\infty}{\lim}|\langle E'(\tilde{u}_n),v\rangle|=\underset{n\rightarrow\infty}{\lim}|\langle E'(u_n),v(\cdot+y_n)\rangle|&=C_{17}\underset{n\rightarrow\infty}{\lim}\|E'(u_n)\|_{X'}\|u_n\|_X\\&=0,~\text{for every}~v\in X.
\end{split}
\end{align}
\end{proof}
\begin{proof}[Proof of Theorem \ref{main_res_1}:]
\begin{enumerate}[label=(\roman*)]
\item By Lemma \ref{mp_geo}, Remark \ref{cerami_rem} and Lemma \ref{alternative}, we can conclude that there exists $u_0$, a critical point of $E$, such that $E(u_0)=d$. 
\item Define $$D=\{v\in X\setminus\{0\}:E'(v)=0\}.$$
Clearly, $D\neq\emptyset$ since $E'(u_0)=0$. This allows us to consider $(u_n)\subset D$, satisfying $E(u_n)\rightarrow d'=\underset{v\in D}{\inf}\{E(v)\}$. Clearly, $d'\in[-\infty,d]$ by the definition of $d'$. Obviously, if $d=d'$ then we already have a solution, namely $u_0$. Also if $d'=0$, then we choose the sequence $(u_n)$ such that $E(u_n)\rightarrow 0$ as $n\rightarrow\infty$ but $\|u_n\|\nrightarrow 0$ since $(a)$ $0$ is not a critical point of $E$, $(b)$ this choice of the sequence $(u_n)$ helps to rule out $(i)$ of Lemma \ref{alternative}.\\
Otherwise if $0\neq d'<d$, there exists a subsequence, still denoted by $(u_n)$, such that $E(u_n)<d$, for all $n$. Thus, Lemma \ref{bdd_seq} holds for this subsequence as well and the results thus derived as a consequence of Lemma \ref{bdd_seq} are all applicable. Furthermore, by using Lemma \ref{mp_geo} and by Lemma \ref{alternative}, there exists $(y_n)\subset\mathbb{Z}$ such that $\tilde{u}_n\rightarrow u$ in $X$ where $u$ turned out to be a nontrivial critical point of $E$ in $X$ by $(ii)$ of Lemma \ref{alternative}. Hence, $E'(u)=\underset{n}{\lim}E'(u_n)=0$. This implies that $u\in D$ and $E(u)=\underset{n}{\lim}E(\tilde{u}_n)=\underset{n}{\lim}E(u_n)=d'>-\infty$.  
\end{enumerate}
\end{proof}
\section{Appendix}
\noindent Lemma \ref{existence_positive_soln} will establish the existence of a positive solution to \eqref{auxprob} and Lemma \ref{u_greater_u_lambda} will guarantee that a solution to \eqref{main} is greater than or equal to the solution to \eqref{auxprob}.
\begin{lemma}[Weak Comparison Principle]\label{weak comparison}
	Let $u, v\in X$. Suppose that, $-\Delta^{\frac{1}{2}}v-\frac{\mu}{v^{\gamma}}\geq-\Delta^{\frac{1}{2}}u-\frac{\mu}{u^{\gamma}}$ weakly in $\mathbb{R}$.
	Then $v\geq u$ in $\mathbb{R}.$
\end{lemma}
\begin{proof}
	The idea was motivated by
	 {\sc Saoudi et al.} \cite{ghosh1}. 
	 Since, $-\Delta^{\frac{1}{2}}v-\frac{\mu}{v^{\gamma}}\geq-\Delta^{\frac{1}{2}}u-\frac{\mu}{u^{\gamma}}$ weakly in $\mathbb{R}$, we have
	{\small\begin{align}\label{compprinci}
		\langle-\Delta^{\frac{1}{2}}v,\phi\rangle-\int_{\mathbb{R}}\frac{\mu\phi}{v^{\gamma}}dx&\geq\langle-\Delta^{\frac{1}{2}}u,\phi\rangle-\int_{\mathbb{R}}\frac{\mu\phi}{u^{\gamma}}dx,
		\end{align}}
	for every ${\phi\geq 0\in X}$. \\
	In particular, choose $\phi=(u-v)^{+}$. Then the inequality in \eqref{compprinci} looks as follows.
	{\small\begin{align}\label{compprinci1}
		&\langle-\Delta^{\frac{1}{2}}v+\Delta^{\frac{1}{2}}u,(u-v)^{+}\rangle-\mu\int_{\mathbb{R}_+}(u-v)^{+}\left(\frac{1}{v^{\gamma}}-\frac{1}{u^{\gamma}}\right)dx\geq 0,
		\end{align}}
	where $\mathbb{R}_+=\{x:u(x)>v(x)\}$.	
	Let $\psi=u-v$. 
	We choose the test function $\phi=(u-v)^{+}$. We express,
	\begin{align}
	\psi&=u-v=(u-v)^{+}-(u-v)^{-}\nonumber
	\end{align}
	to obtain
	\begin{align}\label{negativity}
	[\psi(y)-\psi(x)][\phi(x)-\phi(y)]&=-(\psi^{+}(x)-\psi^{+}(y))^2.
	\end{align}
	The equation in \eqref{negativity} implies
	\begin{align}
	0&\geq \langle  -\Delta^{\frac{1}{2}}v+\Delta^{\frac{1}{2}} u,(v-u)^{+} \rangle=-\int_{\mathbb{R}}\frac{1}{|x-y|^{2}}(\psi^{+}(x)-\psi^{+}(y))^2dxdy\geq0.
	\end{align}
	This leads to the conclusion about the Lebesgue measure of $\mathbb{R}_+$, i.e., $|\mathbb{R}_+|=0$. In other words $v\geq u$ a.e. in $\mathbb{R}$.	 
\end{proof}
\begin{lemma}
	\label{existence_positive_soln}
	Let $\mu>0$. Then the following problem 
	\begin{eqnarray}\label{auxprob_appendix}
	(-\Delta)^{\frac{1}{2}}u+u+(\ln|\cdot|*|u|^2)&=&\mu u^{-\gamma},~\text{in}~\mathbb{R}\nonumber\\
	u&>&0,~\text{in}~\mathbb{R}
	\end{eqnarray}
	has a unique weak solution in $X_0$. This solution is denoted by $\underline{u}_{\mu}$, satisfies $\underline{u}_{\mu}\geq \epsilon_{\mu} v_0$ a.e. in $\Omega$, where $\epsilon_{\mu}>0$ is a constant. Here $v_0>0$ is a suitable function such that $E(\epsilon_{\mu} v_0)<0$.
\end{lemma}
\begin{proof}
	We follow the proof in
	 {\sc Choudhuri} \cite{dc_zamp}
	and
	 {\sc Choudhuri-Saoudi} \cite{chou2}. 
	 First, we note that an energy functional on $X$ formally corresponding to \eqref{auxprob_appendix} can be defined as follows.
	\begin{align}
	\label{ef_aux}
	E(u)&=\frac{1}{2}\|u\|^2+W(u)-\frac{\mu}{1-\gamma}\int_{\Omega}(u^+)^{1-\gamma}dx,
	\end{align}
	for all $u\in X$ where $u^+(x)=\max\{u(x),0\}$. By using the Poincar\'{e} inequality, this functional is coercive and continuous on $X$. It follows that $E$ possesses a global minimizer $u_0\in X$. Obviously, $u_0\neq 0$ since $E(0)=0>E(\epsilon v_0)$, for sufficiently small $\epsilon$ and some $v_0>0$ in $\mathbb{R}$.\\
	Next if $u_0$ is a global minimizer for $E$, then $|u_0|$ is also a global minimizer. This is because $E(|u_0|)\leq E(u_0)$. Clearly, the equality holds if and only if $u_0^-=0$ a.e. in $\mathbb{R}$. Here $u^-(x)=\min\{-u(x),0\}$. In other words we must have $u_0\geq 0$, i.e. $u_0\in X^+$ where 
	$$X^+=\{u\in X:u\geq 0~\text{a.e. in}~\mathbb{R}\}$$
	is the positive cone in $X$.\\
	Furthermore, we will show that $u_0\geq \epsilon v_0>0$ holds a.e. in $\mathbb{R}$ for small enough $\epsilon$. We observe that, 
	\begin{align}\label{neg_der}
	\begin{split}
	\frac{d}{dt}E(tv_0)|_{t=\epsilon}=&\epsilon\|v_0\|^2+4\epsilon^3C(v_0^4)-\mu\epsilon^{-\gamma}\int_{\Omega}v_0^{1-\gamma}dx<0,
	\end{split}
	\end{align}
	whenever $\epsilon\in (0,\epsilon_{\mu}]$, for some sufficiently small $\epsilon_{\mu}$. We now show that $u_0\geq \epsilon_{\mu}v_0$. Suppose we assume the contrary that $w=(\epsilon_{\mu}v_0-u_0)^+$ does not vanish identically in $\mathbb{R}$. We denote $\mathbb{R}_+=\{x\in\mathbb{R}:w(x)>0\}.$
	We will analyye the function $\zeta(t)=E(u_0+tw)$ of $t\geq 0$. This function is convex when defined over $X^+$ being convex. Furthermore $\zeta'(t)=\langle E'(u_0+tw),w \rangle$ is nonnegative and nondecreasing for $t>0$. Consequently, for $0<t<1$ we have
	\begin{align}\label{ineq_appendix}
	\begin{split}
	0\leq \zeta'(1)-\zeta'(t)&=\langle E'(u_0+w)-E'(u_0+tw),w\rangle=\int_{\mathbb{R}_+}E'(u_0+w)dx-\zeta'(t)<0
	\end{split}
	\end{align}
	by inequality \eqref{neg_der} and $\zeta'(t)\geq 0$ with $\zeta'(t)$ being nondecreasing for every $t>0$. which leads to a contradiction. Therefore $w=0$ in $\mathbb{R}$ and hence $u_0\geq \epsilon_{\mu}v_0$ a.e. in $\mathbb{R}$.\\
	Moreover, since the functional $E$ is strictly convex on $X^+$, we conclude that $u_0$ is the only critical point of $E$ in $X^+$ with the property $\underset{V}{\text{ess}\inf}u_0>0$, for any compact subset $V\subset\mathbb{R}$. Thus we choose $\underline{u}_{\mu}=u_0$ in the cutoff functional.
\end{proof}
\begin{remark}\label{obs_pos}  We now perform an analysis on a solution (if it exists). Suppose that $u$ is a solution to \eqref{main}. Then we observe the following
	\begin{enumerate}
		\item If $u$ is a global minimizer, then clearly $E(u)\leq E(|u|)$. Further, $E(u)\geq E(|u|)$ is always true due to the first term of the energy functional. Thus $u^-=0$ a.e. in $\mathbb{R}$.
		\item In fact, a solution to \eqref{main} can be considered to be positive, i.e. $u>0$ a.e. in $\mathbb{R}$, due to the presence of the singular term.
	\end{enumerate}
	Therefore without loss of generality, we may assume that the solution is positive. 
\end{remark}
We finally have the following result.
\begin{lemma}[A priori analysis]\label{u_greater_u_lambda}
	Fix a $\mu\in(0,\mu_0)$. Then a solution of \eqref{main}, say $u>0$, is such that $u> \underline{u}_{\mu}$ a.e. in $\mathbb{R}$.
\end{lemma}
\begin{proof}
	Fix $\mu\in(0,\mu_0)$ and let $u\in X$ be a positive solution to \eqref{main} and $\underline{u}_{\mu}>0$ be a solution to \eqref{auxprob_appendix}. We will show that $u\geq \underline{u}_{\mu}$ a.e. in $\mathbb{R}$. Thus, we let ${\mathbb{R}}^*=\{x\in\mathbb{R}:u(x)<\underline{u}_{\mu}(x)\}$ and from the equation satisfied by $u$, $\underline{u}_{\mu}$, we have 
	\begin{align}\label{comp_1}
	\begin{split}
	0\leq&\langle\Delta^{\frac{1}{2}}(\underline{u}_{\mu}-u),\underline{u}_{\mu}-u\rangle_{\underline{\mathbb{R}}}+\langle W'(\underline{u}_{\mu})-W'(u),\underline{u}_{\mu}-u\rangle+\int_{\mathbb{R}}(f(u)-f(\underline{u}_{\mu}))(\underline{u}_{\mu}-u)dx\\
	\leq&\mu\int_{\mathbb{R}^*}(\underline{u}_{\mu}^{-\gamma}-u^{-\gamma})(\underline{u}_{\mu}-u)dx;~\text{by}~(A_1)\leq 0.
	\end{split}
	\end{align}
	Furthermore, we have 
	\begin{align}\label{comp_2}
	\langle\Delta^{\frac{1}{2}}(\underline{u}_{\mu}-u),\underline{u}_{\mu}-u\rangle_{\mathbb{R}^*}&\geq 0. 
	\end{align}
	Hence, by \eqref{comp_1} and \eqref{comp_2}, we obtain $u\geq\underline{u}_{\mu}$ a.e. in $\mathbb{R}$.\\	
	Now suppose $S=\{x\in\Omega:u(x)=\underline{u}_{\lambda}(x)\}$. Clearly $S$ is a measurable set and hence for any $\delta>0$ there exists a closed subset $F$ of $S$ such that $|S\setminus F|<\delta$. Furthermore, let $|S|>0$. Define a test function $\varphi\in C_c^1(\mathbb{R})$ such that 
	\begin{equation}\varphi(x)=\begin{cases}
	1,& ~\text{if}~ x\in F\\
	0<\varphi<1,&~\text{if}~x\in S\setminus F\\
	0,&~\text{if}~x\in \mathbb{R}\setminus S.
	\end{cases}\end{equation}
	Since $u$ is a weak solution to \eqref{main}, we have 
	\begin{align}\label{equality_breakage}
	\begin{split}
	0=&(\langle (-\Delta)^{\frac{1}{2}}u,\varphi\rangle_{\mathbb{R}}+\int_{F}W'(u) dx+\int_{S\setminus F}W'(u) \varphi dx-\mu\int_{F}u^{-\gamma}dx\\ 
	&-\mu\int_{S\setminus F}u^{-\gamma}\varphi dx-\int_{F}f(u)dx-\int_{S\setminus F}f(u)\varphi dx\\
	=&-\mu\int_{F}u^{-\gamma}dx-\mu\int_{S\setminus F}u^{-\gamma}\varphi dx-\int_{F}f(u)dx-\int_{S\setminus F}f(u)\varphi dx<0.
	\end{split}
	\end{align}
	This is a contradiction. Therefore, $|S|=0$. Hence, $u>\underline{u}_{\lambda}$ a.e. in $\mathbb{R}$.
\end{proof}

\section*{Acknowledgements}
\noindent  
The first author was supported by the Indian Council of Scientific and Industrial Research grant 25(0292)/18/EMR-II. The second author was supported by the Slovenian Research Agency grants P1-0292, N1-0114, N1-0083, N1-0064, and J1-8131.

\end{document}